\documentclass{article}

\usepackage{spiral_layout1}
\usepackage{cite}
\usepackage[utf8]{inputenc} % allow utf-8 input
\usepackage[T1]{fontenc}    % use 8-bit T1 fonts
\usepackage{hyperref}       % hyperlinks
\usepackage{url}            % simple URL typesetting
\usepackage{booktabs}       % professional-quality tables
\usepackage{amsfonts, amssymb}       % blackboard math symbols
\usepackage{nicefrac}       % compact symbols for 1/2, etc.
\usepackage{microtype}      % microtypography
\usepackage{lipsum}
\usepackage{graphicx}
\graphicspath{ {./images/} }
\usepackage{kotex}
\usepackage{amsmath, amsthm}
\newtheorem{theorem}{Theorem}[]
\newtheorem{definition}[theorem]{Definition}

\newtheorem{lemma}[theorem]{Lemma}

\newtheorem{corollary}[theorem]{Corollary}
\newtheorem*{theorem*}{Theorem}

\newlength{\myindent}              
\newcommand{\ind}{\hspace*{\myindent}}

\title{\bf A Logarithmic Spiral Formed by a Sequence of Regular Polygons}

\author{
 Juno Park \\
 Korea Science Academy of KAIST \\
 Busan, Republic of Korea \\[0.7pt]
 \vspace{-3pt}
  \texttt{23-056@ksa.hs.kr} \\
}

\begin{document}

\maketitle
\begin{abstract}
When the sequence of regular polygons with consecutively increasing numbers of sides is joined edge-to-edge in a single direction while minimizing bending, the resulting structure assumes the shape of a logarithmic spiral. This paper proves that this spiral takes the form $r=\exp(4\theta/\pi)$. Specifically, it is derived that the distances between the curve and the centers of the even-sided and odd-sided regular polygons converge to $5/6$ and $7/12$, respectively, with the centers extending outward along the inner side of the spiral. A similar analysis applied to the sequence of regular polygons with consecutively increasing odd numbers of sides reveals that it forms the same type of spiral, establishing that the distances to the centers converge to $7/24$.
\end{abstract}

\section{Introduction}
\begin{figure}[h]
    \centering
    \includegraphics[width=0.45\textwidth]{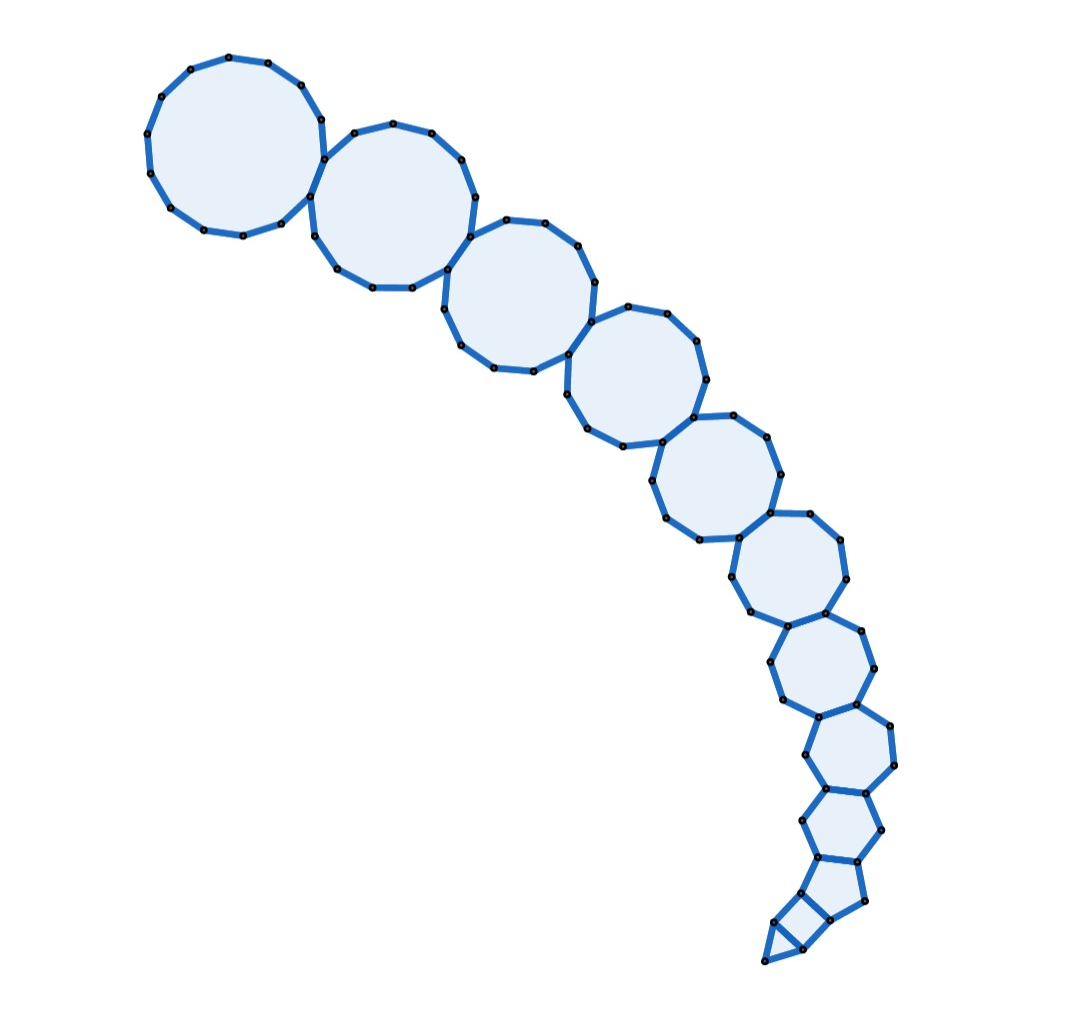}
    \caption{The sequence of regular polygons}
    \label{fig.1}
\end{figure}
Logarithmic spirals have long been a subject of historical study and a source of aesthetic admiration \cite[Chap.~11]{maor2015e}, \cite{thompson1942growth}. Meanwhile, numerous attempts have been made to construct spirals using specific types of polygons \cite{davis1993spirals,Fathauer2021log}, especially regular polygons \cite{fridberg2025regular,mendler2016polygon}. Building upon this, significant research has been dedicated to analytically investigating the properties of spirals generated by such elementary polygons \cite{brink2012spiral,fridberg2025regular}.

\vspace{-1pt}

\ind This paper investigates the logarithmic spiral structure formed by the infinite sequence of regular polygons illustrated in Figure 1. Each regular $(n+1)$-gon shares a side with the regular $n$-gon. The arrangement is constructed to minimize the turning angle between the segment connecting the centers of the regular $(n-1)$-gon and the regular $n$-gon, and the segment connecting the centers of the regular $n$-gon and the regular $(n+1)$-gon. Additionally, the structure extends while bending exclusively to the left. Let the side length of all regular polygons be 1. To analyze this structure, we place it in the complex plane. A sequence corresponding to the centers of the regular polygons is formulated to analyze the curve approximated by the locus of these centers.

\section{Sequence of Centers}

\begin{figure}[h]
    \centering
    \includegraphics[width=0.4\textwidth]{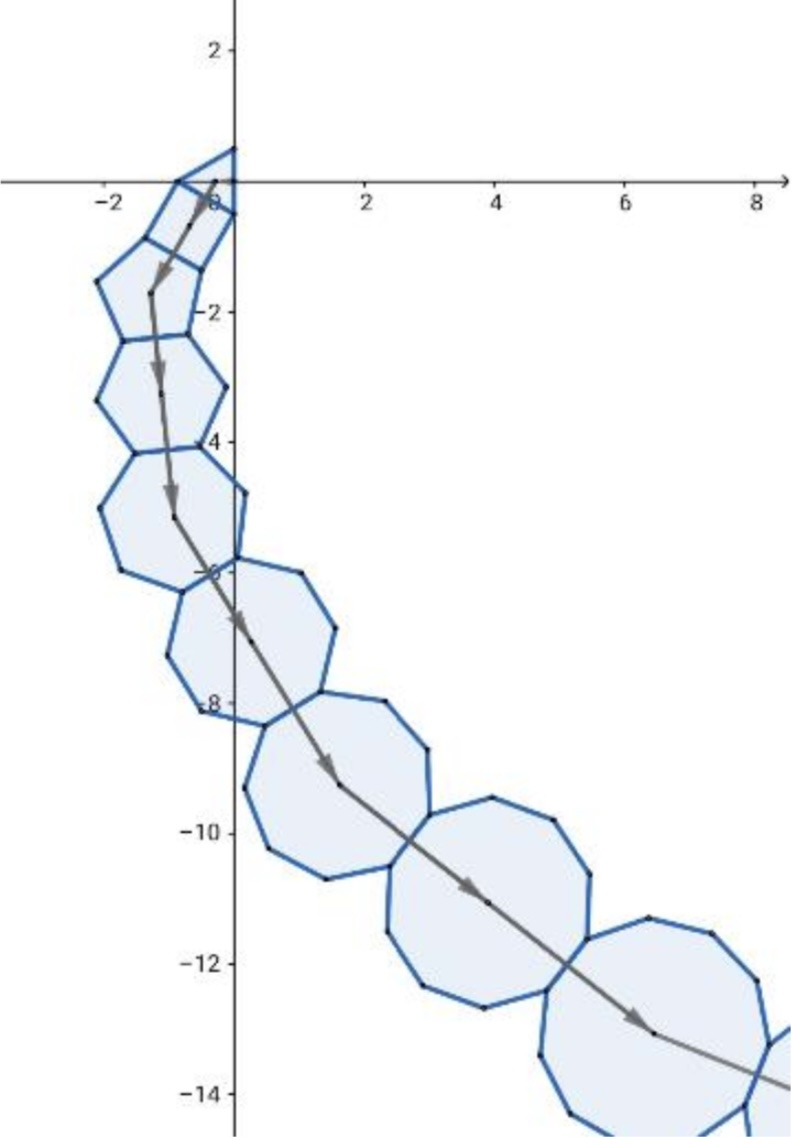}
    \caption{Centers of regular polygons}
    \label{fig.2}
\end{figure}
Figure 2 illustrates the spiral where the midpoint of a side of the equilateral triangle is placed at the origin, and its center is located at $-\frac{\sqrt{3}}{6}+0i$. As observed in the figure, the center of the regular $n$-gon is obtained by adding an appropriate complex number to the center of the regular $(n-1)$-gon. The following sequence is investigated in this context.\\[-6pt]

\begin{definition}
For $n\geq 3$, the sequence of complex numbers $P_n$ is defined as follows:

\vspace{-2pt}

\begin{equation} \label{one}
    P_n=\sum_{k=2}^{n-1}\frac{1}{2}{\left (\cot\left(\frac{\pi}{k}\right)+\cot{\left(\frac{\pi}{k+1}\right)}\right )}\ e^{i\frac{\pi}{2}\left(H_k + h_k \right )}
\end{equation}

where $H_k = \sum_{j=1}^k \frac{1}{j}$ and $h_k = \sum_{j=1}^k \frac{(-1)^{j-1}}{j}$ (this notation is adopted throughout the paper).
\end{definition}

\ind The value $P_n$ defined above coincides with the complex number corresponding to the center of the regular $n$-gon in Figure 2. For a regular $k$-gon with side length 1, the perpendicular distance from the center to a side is $\frac{1}{2}\cot\left(\frac{\pi}{k}\right)$. Hence, the magnitude of the vector connecting the center of the regular $k$-gon to the center of the regular $(k+1)$-gon is $\frac{1}{2}\left(\cot\left(\frac{\pi}{k}\right)+\cot{\left(\frac{\pi}{k+1}\right)}\right)$. On the other hand, the angle of this vector relative to the real axis is given by $\pi, \pi+\frac{\pi}{3}, \pi+\frac{\pi}{3}, \pi+\frac{\pi}{3}+\frac{\pi}{5}, \pi+\frac{\pi}{3}+\frac{\pi}{5}, \pi+\frac{\pi}{3}+\frac{\pi}{5}+\frac{\pi}{7}, \dots$ for $k=3, 4, 5, 6, 7, 8, \dots$, respectively. Thus, the \vadjust{\vspace{2pt}}
general form of the angle is expressed as $\frac{\pi}{2}\left(H_k + h_k \right )$.\\[-10pt]

\ind In the configuration of the spiral shown in Figure 2, while the condition of unit side length is natural, the specific placement of the equilateral triangle (and the consequent placement of the entire spiral) is quite arbitrary. Therefore, the spiral structure is considered equivalent up to rigid transformations; that is, it remains identical under rotation or translation. Consequently, the study primarily focuses on the image of $P_n$ under an appropriate orientation-preserving isometry $T$ in the (complex) plane. \\[-10pt]

\ind In this paper, we establish the following theorem.

\begin{theorem*}
    There exists an orientation-preserving isometry $T$ of the plane such that the distance from $T(P_{n})$ to the curve $r=e^{\frac{4}{\pi}\theta}$ is expressed as
    \[\frac{17}{24}+ \frac{(-1)^{n}}{8}+\mathcal{O}\left(\frac{1}{n}\right)\]
    as $n \to \infty$.
\end{theorem*}

\ind To conveniently handle orientation-preserving isometries and $\mathcal{O}\left({n^{-1}}\right)$ asymptotics, let us define the following equivalence relation.\\[-6pt]

\begin{definition}
    Two sequences of complex numbers $(a_n)$ and $(b_n)$ are said to be equivalent, denoted by $a_n\sim_{\circlearrowleft} b_n$, if there exist constants $\phi \in \mathbb R$ and $c \in \mathbb C$ such that
\[a_n = e^{i\phi}b_n+c+\mathcal{O}\left(\frac1{n}\right).\]
\end{definition}

It is straightforward to verify that $\sim_{\circlearrowleft}$ defines an equivalence relation. Geometrically, $a_n\sim_{\circlearrowleft} b_n$ implies that the two
\vadjust{\vspace{0.5pt}}
sequences are identical up to a rigid transformation, within an error term of order $\mathcal{O}({n^{-1}})$.\\[-5.97pt]

\ind The next section investigates a more tractable sequence equivalent to $P_n$.

\section{Approximating $P_n$}
In this section we will prove that the following relation holds: \\[-1pt]
\[2\pi \left(1+\frac{\pi}{4}i\right) P_n \sim_{\circlearrowleft} \left(n-\frac{1}{2}\right)^{2+i\frac{\pi}{2}} + \left(1+\frac{\pi}{4}i\right)\left(\left(n-\frac{1}{2}\right)^{1+i\frac{\pi}{2}} + \left( \frac{1}{4}+\left(\frac{37}{24}+\frac{(-1)^n}{4}\right)\pi i\right)\left(n-\frac{1}{2}\right)^{i\frac{\pi}{2}}\right).\]
\\[1.5pt]
To prove this, several approximation formulas are required.
\\

\begin{lemma}
\label{lem3}
For the partial sums of the harmonic series, the following expansion holds:
\[H_n = \gamma + \log\left(n+\frac{1}{2}\right) + \frac{1}{24n^2} + \mathcal O\left(\frac{1}{n^3}\right)\]
where $\gamma$ is the Euler-Mascheroni constant.
\end{lemma}
\begin{proof}
The proof relies on the following result by DeTemple \cite{detemple1993}:
\[\frac{1}{24(n+1)^2} < H_n - \gamma - \log\left(n+\frac{1}{2}\right) < \frac{1}{24n^2}. \qedhere\]
\end{proof}

\vspace{5pt}

\begin{lemma}
\label{lem4}
For the partial sums of the alternating harmonic series, the following expansion holds:
\[h_n = \log2 + \frac{(-1)^{n-1}}{2n}+\frac{(-1)^{n}}{4n^2}+ \mathcal O\left(\frac{1}{n^3}\right).\]
\end{lemma}
\begin{proof}
The proof proceeds by considering the even and odd cases separately, using the following identities:
\[h_{2k}=H_{2k}-H_k,\]
\[h_{2k+1}=H_{2k}-H_k + \frac{1}{2k+1}. \qedhere\]
\end{proof}

\vspace{5pt}

\begin{lemma}
\label{lem5}
Suppose that $f : [m, n] \to \mathbb{C}$ is a $C^3$ function and $B_k(x)$ denotes the $k$-th Bernoulli polynomial. Let $S = \sum_{i=m}^n f(i)$ and $I = \int_m^n f(t) \, dt$ . Then,
\begin{equation} \label{two}
    S - I = \frac{f(m)+f(n)}{2} + \int_m^n f'(t) B_1(t-\lfloor t \rfloor) \ dt
\end{equation}
\begin{equation} \label{three}
    = \frac{f(m)+f(n)}{2} + \frac{f'(n)-f'(m)}{12} + \frac{1}{6}\int_m^n f'''(t) B_3(t-\lfloor t \rfloor) \ dt.
\end{equation}
\end{lemma}
\begin{proof}
This corresponds to the Euler-Maclaurin formula for orders 1 and 3 \cite[Eq.~2.10.1]{olver2010nist}.
\end{proof}

\ind Let us now prove the following.\\

\begin{theorem}
\label{thm6}
The following relation holds:
    \[2\pi \left(1+\frac{\pi}{4}i\right) P_n \sim_{\circlearrowleft} \left(n-\frac{1}{2}\right)^{2+i\frac{\pi}{2}} + \left(1+\frac{\pi}{4}i\right)\left(\left(n-\frac{1}{2}\right)^{1+i\frac{\pi}{2}} + \left( \frac{1}{4}+\left(\frac{37}{24}+\frac{(-1)^n}{4}\right)\pi i\right)\left(n-\frac{1}{2}\right)^{i\frac{\pi}{2}}\right).\]
\end{theorem}
\begin{proof}
   \[P_n = \sum_{k=2}^{n-1}\frac{1}{2}{\left (\cot\left(\frac{\pi}{k}\right)+\cot{\left(\frac{\pi}{k+1}\right)}\right )}\ e^{i\frac{\pi}{2}\left(H_k + h_k \right )}\]
   
   \[=\sum_{k=2}^{n-1}\frac{1}{2}{\left (\frac{k}{\pi}+\frac{k+1}{\pi}-\frac{1}{3}\frac{\pi}{k}-\frac{1}{3}\frac{\pi}{k+1}+\mathcal O\left(\frac{1}{k^3}\right)\right )}\ e^{{i\frac{\pi}{2}\left(\gamma+\log\left(k+\frac{1}{2}\right) + \frac{1}{24k^2}+ \log2 - \frac{(-1)^k}{2k} +\frac{(-1)^k}{4k^2}+\mathcal O\left(\frac{1}{k^3}\right)\right )}}\]
(by the Laurent expansion of $\cot$ at zero, and Lemmas 3 and 4)
   
   \[=e^{i\frac{\pi}{2}(\gamma+\log2)} \sum_{k=2}^{n-1}\frac{1}{2}{\left (\frac{k}{\pi}+\frac{k+1}{\pi}-\frac{1}{3}\frac{\pi}{k}-\frac{1}{3}\frac{\pi}{k+1}+\mathcal O\left(\frac{1}{k^3}\right)\right )}\ e^{{i\frac{\pi}{2}\left(- \frac{(-1)^k}{2k} + \left(\frac{1}{24} +\frac{(-1)^k}{4}\right)\frac{1}{k^2}+\mathcal O\left(\frac{1}{k^3}\right)\right )}}\left(k+\frac{1}{2}\right)^{i\frac{\pi}{2}}\]
   
   \[=e^{i\frac{\pi}{2}(\gamma+\log2)} \sum_{k=2}^{n-1}{\left (\frac{k+\frac{1}{2}}{\pi}-\frac{\pi}{3}\frac{1}{k+\frac{1}{2}}+\mathcal O\left(\frac{1}{k^2}\right)\right )}\ e^{{i\frac{\pi}{2}\left(- \frac{(-1)^k}{2}\frac{1}{k+\frac{1}{2}} + \frac{1}{24}\frac{1}{\left(k+\frac{1}{2}\right)^2}+\mathcal O\left(\frac{1}{k^3}\right)\right )}}\left(k+\frac{1}{2}\right)^{i\frac{\pi}{2}}\]
   
\begin{multline*}
=e^{i\frac{\pi}{2}(\gamma+\log 2)} 
  \sum_{k=2}^{n-1}
  \left(\frac{k+\frac12}{\pi}-\frac{\pi}{3}\frac{1}{k+\frac12}+\mathcal O\left(\frac{1}{k^2}\right)\right)\left(1+i\frac{\pi}{2} \left(-\frac{(-1)^k}{2(k+\frac12)} + \frac{1}{24(k+\frac12)^2} + \mathcal O\left(\frac{1}{k^3}\right) \right) \right. \\
\left. + \frac12\left(i\frac{\pi}{2}\right)^2 \left(\frac{(-1)^k}{2(k+\frac12)} \right)^2 + \mathcal O\left(\frac{1}{k^3}\right)\right)\left(k+\frac{1}{2}\right)^{i\frac{\pi}{2}}
\end{multline*}
(by the Taylor expansion of $\exp$ at zero)

   \[\sim_{\circlearrowleft} \sum_{k=2}^{n-1}{\left (\frac{k+\frac{1}{2}}{\pi}-\frac{\pi}{3}\frac{1}{k+\frac{1}{2}}\right )}\left(1- \frac{i\pi(-1)^k}{4}\frac{1}{k+\frac{1}{2}} + \left(\frac{i\pi}{48}-\frac{\pi^2}{32}\right)\frac{1}{\left(k+\frac{1}{2}\right)^2}\right)\left(k+\frac{1}{2}\right)^{i\frac{\pi}{2}}\]
   
\begin{equation} \label{four}
    \sim_{\circlearrowleft} \frac{1}{\pi}\sum_{k=2}^{n-1}\left(k+\frac{1}{2}\right)^{1+i\frac{\pi}{2}} -\frac{i}{4}\sum_{k=2}^{n-1}(-1)^k\left(k+\frac{1}{2}\right)^{i\frac{\pi}{2}} + \left(-\frac{35\pi}{96}+\frac{i}{48}\right)\sum_{k=2}^{n-1}\left(k+\frac{1}{2}\right)^{-1+i\frac{\pi}{2}}.
\end{equation}

Set
\[S_1(n)=\sum_{k=2}^{n-1}\left(k+\frac{1}{2}\right)^{1+i\frac{\pi}{2}},\]

\[S_0(n)=\sum_{k=2}^{n-1}(-1)^k\left(k+\frac{1}{2}\right)^{i\frac{\pi}{2}},\]
and
\[S_{-1}(n)=\sum_{k=2}^{n-1}\left(k+\frac{1}{2}\right)^{-1+i\frac{\pi}{2}}.\]

Applying \eqref{three} of Lemma 5 to $S_1(n)$, we obtain

\[S_1(n)=\sum_{k=2}^{n-1}\left(k+\frac{1}{2}\right)^{1+i\frac{\pi}{2}}\]

\[\sim_{\circlearrowleft}\int_2^{n-1}\left(t+\frac{1}{2}\right)^{1+i\frac{\pi}{2}}\ dt+\frac{1}{2}\left(n-\frac{1}{2}\right)^{1+i\frac{\pi}{2}}+\frac{1+i\frac{\pi}{2}}{12}\left(n-\frac{1}{2}\right)^{i\frac{\pi}{2}}\]
\[+\frac{\left(1+i\frac{\pi}{2}\right)\left(i\frac{\pi}{2}\right)\left(-1+i\frac{\pi}{2}\right)}{6}\int_2^{n-1} \left(t-\frac{1}{2}\right)^{-2+i\frac{\pi}{2}} B_3(t-\lfloor t \rfloor) \ dt.\]

Let $M=\sup_{t \in [0,1]}|B_3(t)|$. Then
\[\int_2^{n-1} \left| \left(t-\frac{1}{2}\right)^{-2+i\frac{\pi}{2}} B_3(t-\lfloor t \rfloor) \right| \ dt \leq M\int_2^{n-1} \left(t-\frac{1}{2}\right)^{-2} \ dt\leq 0.4M,\]
which implies the integral $\int_2^{n-1} \left(t-\frac{1}{2}\right)^{-2+i\frac{\pi}{2}} B_3(t-\lfloor t \rfloor) \ dt$ converges as $n \to \infty$. The rate of convergence is $\mathcal{O}(n^{-1})$, since
\[\left | \int_{n-1}^{\infty} \left(t-\frac{1}{2}\right)^{-2+i\frac{\pi}{2}} B_3(t-\lfloor t \rfloor) \ dt \right | \leq \int_{n-1}^{\infty} \left | \left(t-\frac{1}{2}\right)^{-2+i\frac{\pi}{2}} B_3(t-\lfloor t \rfloor) \right | \ dt\]

\[\leq M\int_{n-1}^{\infty} \left(t-\frac{1}{2}\right)^{-2} \ dt=\frac{M}{n-\frac{3}{2}}=\mathcal O\left(\frac{1}{n}\right).\]

Hence
\[\int_2^{n-1} \left(t-\frac{1}{2}\right)^{-2+i\frac{\pi}{2}} B_3(t-\lfloor t \rfloor) \ dt \sim_{\circlearrowleft} 0\]

and it follows that
\begin{equation} \label{five}
    S_1(n)\sim_{\circlearrowleft} \frac{1}{2+i\frac{\pi}{2}}\left(n-\frac{1}{2}\right)^{2+i\frac{\pi}{2}}+\frac{1}{2}\left(n-\frac{1}{2}\right)^{1+i\frac{\pi}{2}}+\frac{1+i\frac{\pi}{2}}{12}\left(n-\frac{1}{2}\right)^{i\frac{\pi}{2}}.
\end{equation}

Analogously, use \eqref{two} for $S_{-1}(n)$, we have
\[S_{-1}(n) \sim_{\circlearrowleft} \int_2^{n-1}\left(t+\frac{1}{2}\right)^{-1+i\frac{\pi}{2}}\ dt+\frac{1}{2}\left(n-\frac{1}{2}\right)^{-1+i\frac{\pi}{2}}+
\left(-1+i\frac{\pi}{2}\right)\int_2^{n-1} \left(t+\frac{1}{2}\right)^{-2+i\frac{\pi}{2}} B_1(t-\lfloor t \rfloor) \ dt\]
\begin{equation} \label{six}
    \sim_{\circlearrowleft} \frac{1}{i\frac{\pi}{2}}\left(n-\frac{1}{2}\right)^{i\frac{\pi}{2}}.
\end{equation}

A separate approach is required to simplify $S_0(n)$.

Let $n=2m+1$. Then,

\[S_0(2m+1)=\sum_{k=2}^{2m}(-1)^k\left(k+\frac{1}{2}\right)^{i\frac{\pi}{2}}\]

\[= \sum_{k=1}^m \left(2k+\frac{1}{2}\right)^{i\frac{\pi}{2}}-\sum_{k=2}^m \left((2k-1)+\frac{1}{2}\right)^{i\frac{\pi}{2}}\]

\[= \sum_{k=1}^m \left(2k+\frac{1}{2}\right)^{i\frac{\pi}{2}}-\sum_{k=2}^m \left(2k-\frac{1}{2}\right)^{i\frac{\pi}{2}}\]

\[\sim_{\circlearrowleft} \int_1^m \left(2t+\frac{1}{2}\right)^{i\frac{\pi}{2}}dt +\frac{1}{2}\left(2m+\frac{1}{2}\right)^{i\frac{\pi}{2}}- \int_2^m \left(2t-\frac{1}{2}\right)^{i\frac{\pi}{2}}dt - \frac{1}{2}\left(2m-\frac{1}{2}\right)^{i\frac{\pi}{2}}\]

\[\sim_{\circlearrowleft} \frac{1}{2}\frac{1}{1+i\frac{\pi}{2}}\left(\left(2m+\frac{1}{2}\right)^{1+i\frac{\pi}{2}}-\left(2m-\frac{1}{2}\right)^{1+i\frac{\pi}{2}}\right)  +\frac{1}{2}\left(\left(2m+\frac{1}{2}\right)^{i\frac{\pi}{2}}-\left(2m-\frac{1}{2}\right)^{i\frac{\pi}{2}}\right)\]

\[= \frac{1}{2}\frac{1}{1+i\frac{\pi}{2}}\left(\left(n-\frac{1}{2}\right)^{1+i\frac{\pi}{2}}-\left(n-\frac{3}{2}\right)^{1+i\frac{\pi}{2}}\right)  +\frac{1}{2}\left(\left(n-\frac{1}{2}\right)^{i\frac{\pi}{2}}-\left(n-\frac{3}{2}\right)^{i\frac{\pi}{2}}\right)\]

(by \eqref{three}). Similarly, for $n=2m$,
\[S_0(2m)=\sum_{k=2}^{2m-1}(-1)^k\left(k+\frac{1}{2}\right)^{i\frac{\pi}{2}}\]

\[= \sum_{k=1}^{m-1} \left(2k+\frac{1}{2}\right)^{i\frac{\pi}{2}}-\sum_{k=2}^m \left((2k-1)+\frac{1}{2}\right)^{i\frac{\pi}{2}}\]

\[\sim_{\circlearrowleft} \frac{1}{2}\frac{1}{1+i\frac{\pi}{2}}\left(\left(2(m-1)+\frac{1}{2}\right)^{1+i\frac{\pi}{2}}-\left(2m-\frac{1}{2}\right)^{1+i\frac{\pi}{2}}\right)  +\frac{1}{2}\left(\left(2(m-1)+\frac{1}{2}\right)^{i\frac{\pi}{2}}-\left(2m-\frac{1}{2}\right)^{i\frac{\pi}{2}}\right)\]

\[= \frac{1}{2}\frac{1}{1+i\frac{\pi}{2}}\left(\left(n-\frac{3}{2}\right)^{1+i\frac{\pi}{2}}-\left(n-\frac{1}{2}\right)^{1+i\frac{\pi}{2}}\right)  +\frac{1}{2}\left(\left(n-\frac{3}{2}\right)^{i\frac{\pi}{2}}-\left(n-\frac{1}{2}\right)^{i\frac{\pi}{2}}\right).\]

\vspace{1pt}
Let $h(t)=t^{1+i\frac{\pi}{2}}-(t-1)^{1+i\frac{\pi}{2}}$. Then for sufficiently large real $t$,
\vspace{1pt}
\[h(t)=t^{1+i\frac{\pi}{2}}\left(1-\left(1-\frac{1}{t}\right)^{1+i\frac{\pi}{2}}\right)\]

\[=t^{1+i\frac{\pi}{2}}\left(1-\left(1-\left(1+i\frac{\pi}{2}\right)\frac{1}{t}+\mathcal O\left(\frac{1}{t^2}\right)\right)\right)\]
(by the binomial series)

\[=\left(1+i\frac{\pi}{2}\right)t^{i\frac{\pi}{2}}+\mathcal O\left(\frac{1}{t}\right).\]

Likewise, one can show that $t^{i\frac{\pi}{2}}-(t-1)^{i\frac{\pi}{2}}=\mathcal{O}({t^{-1}})$. Thus,

\begin{equation} \label{seven}
    S_0(n) \sim_{\circlearrowleft} -\frac{(-1)^n}{2}\left(n-\frac{1}{2}\right)^{i\frac{\pi}{2}}.
\end{equation}

Consequently, applying \eqref{five}, \eqref{six}, and \eqref{seven} to \eqref{four},
\[P_n\sim_{\circlearrowleft} \frac{1}{\pi}S_1(n) -\frac{i}{4}S_0(n) + \left(-\frac{35\pi}{96}+\frac{i}{48}\right)S_{-1}(n)\]

\[\sim_{\circlearrowleft} \frac{1}{\pi}\left( \frac{1}{2+i\frac{\pi}{2}}\left(n-\frac{1}{2}\right)^{2+i\frac{\pi}{2}}+\frac{1}{2}\left(n-\frac{1}{2}\right)^{1+i\frac{\pi}{2}}+\frac{1+i\frac{\pi}{2}}{12}\left(n-\frac{1}{2}\right)^{i\frac{\pi}{2}}\right)\]
\[+\frac{i}{4}\frac{(-1)^n}{2}\left(n-\frac{1}{2}\right)^{i\frac{\pi}{2}} + \left(-\frac{35\pi}{96}+\frac{i}{48}\right)\frac{1}{i\frac{\pi}{2}}\left(n-\frac{1}{2}\right)^{i\frac{\pi}{2}}\]

\[\sim_{\circlearrowleft}\frac{1}{2\pi}\frac{1}{1+i\frac{\pi}{4}}\left(n-\frac{1}{2}\right)^{2+i\frac{\pi}{2}} + \frac{1}{2\pi}\left(n-\frac{1}{2}\right)^{1+i\frac{\pi}{2}} + \left( \frac{1}{8\pi}+\left(\frac{37}{48}+\frac{(-1)^n}{8}\right) i\right)\left(n-\frac{1}{2}\right)^{i\frac{\pi}{2}}\]

and it follows that
\[2\pi \left(1+\frac{\pi}{4}i\right) P_n \sim_{\circlearrowleft} \left(n-\frac{1}{2}\right)^{2+i\frac{\pi}{2}} + \left(1+\frac{\pi}{4}i\right)\left(\left(n-\frac{1}{2}\right)^{1+i\frac{\pi}{2}} + \left( \frac{1}{4}+\left(\frac{37}{24}+\frac{(-1)^n}{4}\right)\pi i\right)\left(n-\frac{1}{2}\right)^{i\frac{\pi}{2}}\right).\]
\end{proof}

\section{The Logarithmic Spiral}

We now investigate a spiral curve that maintains an asymptotic distance from $T(P_n)$ for a suitable isometry $T$.\\[-5pt]

\begin{lemma}
\label{lem7}
   Let $t>0$ and $a, b \in \mathbb{R}$. Define
\vspace{1pt}
    \begin{equation} \label{nine}
        f(t) =  t^{2+i\frac{\pi}{2}} + \left(1+\frac{\pi}{4}i\right)\left(t^{1+i\frac{\pi}{2}} + \left( a +b\frac{\pi}{4}i\right)t^{i\frac{\pi}{2}}\right).
    \end{equation}
 
    Then, the following limit holds with a convergence rate of $\mathcal{O}(t^{-1})$:

\vspace{-4pt}
    
     \[\lim_{t \to \infty} \left ( |f(t)| - e^{\frac{4}{\pi}\arg f(t)} \right ) = \left( \frac{1}{2} -b\right)\left( 1+\frac{\pi^2}{16}\right)\]

where $\arg f(t)$ is the continuous branch satisfying $\arg f(t) = \frac{\pi}{2} \log t +o(1)$ as $t \to \infty$.
\end{lemma}

\begin{proof}
    Let $\epsilon(t)= \left(1+\frac{\pi}{4}i\right) \left(\frac{1}{t}+\frac{a+b\frac{\pi}{4}i}{t^2}\right)$. Note that $\epsilon(t)=\mathcal{O}(t^{-1})$, and we can write
    
    \[f(t)=t^{2+i\frac{\pi}{2}}(1+\epsilon).\]
    Thus,
    \vspace{5pt}
    \[|f(t)|=t^2\sqrt{1+2\Re(\epsilon)+|\epsilon|^2}\]
    
    \[=t^2\left(1+\Re(\epsilon)+\frac{|\epsilon|^2}{2}-\frac{\Re(\epsilon)^2}{2}+\mathcal O\left(\epsilon^3\right)\right)\]
(by the binomial series)
    
    \[=t^2\left(1+\Re(\epsilon)+\frac{|\epsilon|^2}{2}-\frac{\Re(\epsilon)^2}{2}\right)+\mathcal O\left(\frac{1}{t}\right).\]
    
    Since
    \[\Re(\epsilon)=\frac{1}{t}+\frac{a-\frac{\pi^2}{16}b}{t^2}+\mathcal O\left(\frac{1}{t^3}\right) \quad \text{and} \quad |\epsilon|^2=\epsilon\bar{\epsilon}=\frac{1+\frac{\pi^2}{16}}{t^2}+\mathcal O\left(\frac{1}{t^3}\right),\]
    
    it follows that
    \[|f(t)|=t^2\left(1+\frac{1}{t}+\frac{1}{t^2}\left(a-\frac{\pi^2}{16}b+\frac{\pi^2}{32}\right)\right)+\mathcal O\left(\frac{1}{t}\right)\]
    \[=t^2+t+a-\frac{\pi^2}{16}b+\frac{\pi^2}{32}+\mathcal O\left(\frac{1}{t}\right).\]

    On the other hand, for sufficiently large $t$ such that $|\epsilon|<1$, we have
    
    \[\arg f(t) = \frac{\pi}{2}\log t + \mathrm{Arg} (1+\epsilon)\]
    
    \[=\frac{\pi}{2}\log t + \arctan \left(\frac{\Im(\epsilon)}{1+\Re(\epsilon)}\right)\]
    
    \[=\frac{\pi}{2}\log t + \arctan \left({\Im(\epsilon)}(1-\Re(\epsilon))+ \mathcal{O}(\epsilon^3) \right)\]
(by the series expansion of $\frac{1}{1+x}$ at zero)
    \[=\frac{\pi}{2}\log t +\Im(\epsilon)(1-\Re(\epsilon))+\mathcal{O}(\epsilon^3)\]
(by the series expansion of $\arctan$ at zero)
    
    \[=\frac{\pi}{2}\log t +\Im(\epsilon)-\Re(\epsilon)\Im(\epsilon)+\mathcal{O}(\epsilon^3)\]
    
    \[=\frac{\pi}{2}\log t +\frac{\pi}{4t}+\frac{\pi(a+b-1)}{4t^2}+\mathcal O\left(\frac{1}{t^3}\right).\]
    
    Consequently,\\[6pt]
    \[e^{\frac{4}{\pi}\arg f(t)}=e^{2\log t +\frac{1}{t}+\frac{(a+b-1)}{t^2}+\mathcal O\left(\frac{1}{t^3}\right)}\]
    
    \[=t^2\left(1+\frac{1}{t}+\frac{a+b-1}{t^2}+\frac{1}{2t^2}\right)+\mathcal O\left(\frac{1}{t}\right)\]
    
    \[=t^2+t+a+b-\frac{1}{2}+\mathcal O\left(\frac{1}{t}\right).\]
    
    Finally, subtracting the two results yields
    \[|f(t)| - e^{\frac{4}{\pi}\arg f(t)} = -\frac{\pi^2}{16}b+\frac{\pi^2}{32}-b+\frac{1}{2}+\mathcal O\left(\frac{1}{t}\right)\]
    \begin{equation*}
        =\left( \frac{1}{2} -b\right)\left( 1+\frac{\pi^2}{16}\right)+\mathcal O\left(\frac{1}{t}\right),
    \end{equation*}
    
    which completes the proof.
\end{proof}

\begin{corollary}
\label{cor8}
    The distance between $f(n-\frac{1}{2})$ and the curve $r=e^{\frac{4}{\pi}\theta}-\left(b- \frac{1}{2}\right)\left( 1+\frac{\pi^2}{16}\right)$ is $\mathcal O(n^{-1})$.
\end{corollary}
\begin{proof}
    Since
    \[|f(n-\frac{1}{2})| - e^{\frac{4}{\pi}\arg f(n-\frac{1}{2})} = \left( \frac{1}{2} -b\right)\left( 1+\frac{\pi^2}{16}\right)+\mathcal O\left(\frac{1}{n}\right),\]
    it follows that each $f(n-\frac{1}{2})$ lies on a curve
    \[r=e^{\frac{4}{\pi}\theta}-\left(b- \frac{1}{2}\right)\left( 1+\frac{\pi^2}{16}\right)+c_n\]
    for some $c_n = \mathcal O(n^{-1})$. Hence the distance from $f(n-\frac{1}{2})$ to the curve $r=e^{\frac{4}{\pi}\theta}-\left(b- \frac{1}{2}\right)\left( 1+\frac{\pi^2}{16}\right)$ is bounded by $|c_n|$, and thus is $\mathcal O(n^{-1})$.
\end{proof}

\vspace{4pt}

\ind We are now ready to prove the following result.\\[-4pt]

\begin{theorem}
\label{thm9}
There exists an orientation-preserving isometry $T$ of the plane such that the distances from $T(P_{2n})$ and $T(P_{2n+1})$ to the curves
\vspace{-6pt}

\[
r = e^{\frac{4}{\pi}\theta} - \frac{10}{3\pi}\sqrt{1+\frac{\pi^2}{16}}
\quad \text{and} \quad
r = e^{\frac{4}{\pi}\theta} - \frac{7}{3\pi}\sqrt{1+\frac{\pi^2}{16}},
\]

respectively, are $\mathcal{O}(n^{-1})$ as $n \to \infty$.
\end{theorem}
\begin{proof}
        
Let
\vspace{-2pt}
\begin{align*}
    g_{e}(t) &= t^{2+i\frac{\pi}{2}} + \left(1+\frac{\pi}{4}i\right)\left(t^{1+i\frac{\pi}{2}} + \left( \frac{1}{4}+\frac{43}{24}\pi i\right)t^{i\frac{\pi}{2}}\right), \\[1pt]
    g_{o}(t) &= t^{2+i\frac{\pi}{2}} + \left(1+\frac{\pi}{4}i\right)\left(t^{1+i\frac{\pi}{2}} + \left( \frac{1}{4}+\frac{31}{24}\pi i\right)t^{i\frac{\pi}{2}}\right),
\end{align*}

which are the cases of \eqref{nine} with $b=\frac{43}{6}$ and $b=\frac{31}{6}$, respectively. Then, by Lemma 7, $g_e(n-\frac{1}{2})$ and $g_o(n-\frac{1}{2})$ satisfy
\begin{align*}
    |g_e(n-\tfrac{1}{2})| - e^{\frac{4}{\pi}\arg g_e(n-\frac{1}{2})} &= -\frac{20}{3}\left( 1+\frac{\pi^2}{16}\right)+\mathcal O\left(\frac{1}{n}\right), \\[2pt]
    |g_o(n-\tfrac{1}{2})| - e^{\frac{4}{\pi}\arg g_o(n-\frac{1}{2})} &= -\frac{14}{3}\left( 1+\frac{\pi^2}{16}\right)+\mathcal O\left(\frac{1}{n}\right),
\end{align*}
respectively. Also, by Theorem 6,
    
    \[2\pi \left(1+\frac{\pi}{4}i\right) P_n \sim_{\circlearrowleft} \begin{cases} 
        g_{e}\left(n-\frac{1}{2}\right) & \text{if $n$ is even} \\[4pt]
        g_{o}\left(n-\frac{1}{2}\right) & \text{if $n$ is odd}
     \end{cases}\]

which implies that
\vspace{5pt}
\[T'\left( 2\pi \left(1+\frac{\pi}{4}i\right) P_n \right) = \begin{cases} 
       g_{e}\left(n-\frac{1}{2}\right) + \mathcal O\left(\frac{1}{n}\right) & \text{if $n$ is even} \\[4pt]
        g_{o}\left(n-\frac{1}{2}\right)+\mathcal O\left(\frac{1}{n}\right) &
        \text{if $n$ is odd}
     \end{cases}\]

holds for some orientation-preserving isometry $T'$ of the (complex) plane.
\\

Suppose that $n$ is even. Then $T'( 2\pi \left(1+\frac{\pi}{4}i\right) P_n )$ is within distance $\mathcal O(n^{-1})$ of $g_{e}(n-\frac{1}{2})$. Since $g_{e}(n-\frac{1}{2})$ is within distance $\mathcal O(n^{-1})$ of the point on $r=e^{\frac{4}{\pi}\theta}-\frac{20}{3}( 1+\frac{\pi^2}{16})$ by Corollary 8, the shortest distance from $T'( 2\pi \left(1+\frac{\pi}{4}i\right) P_n )$ to the curve $r=e^{\frac{4}{\pi}\theta}-\frac{20}{3}( 1+\frac{\pi^2}{16})$ is $\mathcal O(n^{-1})$. Therefore, the point

\[\frac{1}{2\pi\sqrt{1+\frac{\pi^2}{16}}} \ T'\left( 2\pi \left(1+\frac{\pi}{4}i\right) P_n \right)\]

is within distance $\mathcal O(n^{-1})$ of the curve

\[2\pi\sqrt{1+\frac{\pi^2}{16}} \ r=e^{\frac{4}{\pi}\theta}-\frac{20}{3}\left( 1+\frac{\pi^2}{16}\right)\]

\[\iff r=e^{\frac{4}{\pi}\left(\theta-\frac{\pi}{4}\log\left(2\pi\sqrt{1+\frac{\pi^2}{16}}\right)\right)}-\frac{10}{3\pi}\sqrt{1+\frac{\pi^2}{16}}.\]

Thus, by using the isometry of $\mathbb C$ defined as

\begin{equation} \label{eight}
    T(z)=\frac{1}{\left(2\pi\sqrt{1+\frac{\pi^2}{16}}\right)^{1+i\frac{\pi}{4}}} \ T'\left( 2\pi \left(1+\frac{\pi}{4}i\right) z \right),
\end{equation}

we find that the distance from $T(P_n)$ to the curve $r = e^{\frac{4}{\pi}\theta} - {\frac{10}{3\pi}\sqrt{1+\frac{\pi^2}{16}}}$ is $\mathcal O(n^{-1})$. The case where $n$ is odd is proved in the same way using the same $T$, since

\[\frac{\frac{14}{3}\left( 1+\frac{\pi^2}{16}\right)}{2\pi\sqrt{1+\frac{\pi^2}{16}}} = \frac{7}{3\pi}\sqrt{1+\frac{\pi^2}{16}}. \qedhere\]
\end{proof}

\vspace{2pt}
\ind The next theorem gives the distance between a point on $r = e^{\frac{4}{\pi}{\theta}}-c$ and the curve $r = e^{\frac{4}{\pi}{\theta}}$.\\[-6pt]

\begin{theorem}
\label{thm10}
    Let $C_1$ be the curve given by $r=e^{\beta \theta}$ and $C_2$ be the curve given by $r=e^{\beta \theta}-c$ (where $\beta, c \geq 0$). For a point $Q(r,\theta)$ on $C_2$, the distance $d(r)$ from $Q$ to the nearest point $P$ on $C_1$ is given by
    \vspace{2pt}
    \[d(r)=\frac{c}{\sqrt{1+\beta^2}}+\mathcal{O}\left(\frac{1}{r}\right) \quad \text{as } r \to \infty.\]
\end{theorem}
\begin{proof}
    First, let $O$ be the origin and assume $r$ is sufficiently large. The vector $\overrightarrow{PQ}$ is perpendicular to the tangent of $C_1$ at $P$. Let $\mathbf e_r$ and $\mathbf e_\theta$ be the polar unit vectors at $P$. Since $\nabla(r-e^{\beta\theta})=\mathbf e_r -\beta \mathbf e_\theta$, we have $\overrightarrow{PQ}=-d\cdot \frac{\mathbf e_r -\beta \mathbf e_\theta}{\sqrt{1+\beta^2}}$. Let the coordinates of $P$ be $(r_1,\theta_1)$. Then $\overrightarrow{OP}=r_1\mathbf{e}_r$, and
    \[\overrightarrow{OQ}=\overrightarrow{OP}+\overrightarrow{PQ}=\left( r_1 - \frac{d}{\sqrt{1+\beta^2}}\right)\mathbf{e}_r + \left(\frac{\beta d}{\sqrt{1+\beta^2}}\right)\mathbf{e}_\theta.\]
    Thus, for $Q(r_2,\theta_2)$, we have

    \[r_2^2=\left( r_1 - \frac{d}{\sqrt{1+\beta^2}}\right)^2 + \left(\frac{\beta d}{\sqrt{1+\beta^2}}\right)^2\]
    
    \[=r_1^2 - 2r_1 \frac{d}{\sqrt{1+\beta^2}}+d^2.\]
 
    Note that $d$ is bounded with respect to $r_1$ since $d(r_2)\leq c$ for all $r_2$ by the definition of $d$. Therefore,
\vspace{1.2pt}
    \[r_2 =r_1 \left( 1- 2\frac{d}{\sqrt{1+\beta^2}}\frac{1}{r_1}+\frac{d^2}{r_1^2} \right)^{\frac{1}{2}}\]
    
    \[=r_1 \left( 1- \frac{d}{\sqrt{1+\beta^2}} \frac{1}{r_1}+ \mathcal O\left(\frac{1}{r_1^2}\right) \right)\]
    (by the binomial series)

    \[=r_1 - \frac{d}{\sqrt{1+\beta^2}}+ \mathcal O\left(\frac{1}{r_1}\right). \]
\\

On the other hand,
\[\tan(\theta_2-\theta_1)=\frac{\frac{\beta d}{\sqrt{1+\beta^2}}}{ r_1 - \frac{d}{\sqrt{1+\beta^2}}}\]

and since $d$ is bounded and $r_1$ is sufficiently large, the RHS is close to $0$ (thus lies in $(-\frac{\pi}{2},\frac{\pi}{2})$), so

\[\theta_2-\theta_1 = \arctan\left(\frac{\frac{\beta d}{\sqrt{1+\beta^2}}}{ r_1 - \frac{d}{\sqrt{1+\beta^2}}}\right).\]

Therefore,

\[r_2 = e^{\beta\theta_2}-c = r_1e^{\beta(\theta_2-\theta_1)}-c\]

\[=r_1\exp \left(\beta \arctan\left(\frac{\frac{\beta d}{\sqrt{1+\beta^2}}}{ r_1 - \frac{d}{\sqrt{1+\beta^2}}}\right)\right)-c\]

\[=r_1\left(1+\beta  \frac{\frac{\beta d}{\sqrt{1+\beta^2}}}{ r_1 - \frac{d}{\sqrt{1+\beta^2}}} +\mathcal{O}\left(\frac{1}{r_1^2}\right) \right)-c\]
(by $\exp(\beta \arctan x) =1+\beta x +\mathcal O({x^2})$)

\[=r_1+\frac{\beta^2 d}{\sqrt{1+\beta^2}}-c+\mathcal{O}\left(\frac{1}{r_1}\right).\]

Thus, we have

 \[r_1 - \frac{d}{\sqrt{1+\beta^2}} =r_1+\frac{\beta^2 d}{\sqrt{1+\beta^2}}-c+\mathcal{O}\left(\frac{1}{r_1}\right),\]

 which implies

\[c=d\sqrt{1+\beta^2} + \mathcal{O}\left(\frac{1}{r_1}\right).\]

Therefore,

\[d(r)=\frac{c}{\sqrt{1+\beta^2}} + \mathcal{O}\left(\frac{1}{r}\right). \qedhere\]
  
\end{proof}

\vspace{10pt}

\ind Finally, we have arrived at the following theorem.\\[-4pt]

\begin{theorem}
 There exists an orientation-preserving isometry $T$ of the plane such that the distance from $T(P_{n})$ to the curve $r=e^{\frac{4}{\pi}\theta}$ is given by
    \[\frac{17}{24}+ \frac{(-1)^{n}}{8}+\mathcal{O}\left(\frac{1}{n}\right)\]
    as $n \to \infty$.
\end{theorem}

\begin{proof}

Let $C$ be the set of points in the plane satisfying $r=e^{\frac{4}{\pi}\theta}$, and let $d(T(P_n),C)$ denote the distance from the\\
point $T(P_n)$ to the set $C$. First, we prove that $d(T(P_{2m}),C)=\frac{5}{6}+\mathcal{O}(m^{-1})$.

Suppose that $n$ is even. Then, for any point $p$ on the curve $r=e^{\frac{4}{\pi}\theta} - {\frac{10}{3\pi}\sqrt{1+\frac{\pi^2}{16}}}$ and any point $q$ on the curve $r=e^{\frac{4}{\pi}\theta}$, the following triangle inequality holds:

\[|T(P_n)-p|+|p-q| \geq |T(P_n)-q| \geq d(T(P_n),C)\]

(where $T$ is defined as in \eqref{eight}). Let $p_n$ be a point on $r=e^{\frac{4}{\pi}\theta} - {\frac{10}{3\pi}\sqrt{1+\frac{\pi^2}{16}}}$ satisfying $|T(P_n)-p_n|\leq \frac{A}{n}$ (the existence of such a point for some appropriate constant $A>0$ is guaranteed by Theorem 9). Also, let $q_n$ be a point on $C$ that minimizes the distance to $p_n$ (such a point exists since $C$ is a closed set). Then, by Theorem 10, we have

\[\left||p_n-q_n|-\frac{5}{6}\right|\leq \frac{B}{|p_n|}\]

for some constant $B>0$, since

\[\frac{\frac{10}{3\pi}\sqrt{1+\frac{\pi^2}{16}}}{\sqrt{1+\left(\frac{4}{\pi}\right)^2}} = \frac{5}{6}.\]

Note that $\frac{B}{|p_n|}\leq\frac{B}{n}$ holds, because

\[|p_n| \geq |T(P_n)|-\frac{A}{n}\geq |P_n|-c - \frac{A}{n} \geq \frac{1}{\left | 2\pi\left(1+\frac{\pi}{4}i\right) \right |}\left| n- \frac{1}{2}\right|^2 - \frac{1}{2\pi} \left|n-\frac{1}{4}+\frac{43}{24}\pi i  \right|-c-\frac{A}{n} \geq n\]

for sufficiently large $n$ (where $c$ is a constant accounting for the translational component of the isometry $T$).

Consequently,
\[d(T(P_n),C) \leq  \frac{A}{n} +\frac{5}{6} +\frac{B}{n}\]

for all sufficiently large $n$.

Now, let $\tilde{q}_n$ be a point on $C$ satisfying $|T(P_n)-\tilde{q}_n|=d(T(P_n),C)$ (again, existence is guaranteed since $C$ is a closed set). Then,

\[d(T(P_n),C) = |\tilde{q}_n-T(P_n)| \geq |\tilde{q}_n-p_n| - |p_n-T(P_n)|\]
\[\geq |q_n-p_n| - |p_n-T(P_n)| \geq \frac{5}{6}-\frac{B}{n}-|p_n-T(P_n)|\]
\[\geq \frac{5}{6}-\frac{B}{n}-\frac{A}{n}\]

for all sufficiently large $n$. Thus, since $\left|d(T(P_n),C)-\frac{5}{6}\right| \leq \frac{A}{n}+\frac{B}{n}$ for large $n$, it follows that for even $n$,

\vspace{-3pt}

\[d(T(P_n),C)=\frac{5}{6}+\mathcal{O}\left(\frac{1}{n}\right).\]

For odd $n$, the proposition is proved by the same argument (using the same $T$), noting that

\[\frac{\frac{7}{3\pi}\sqrt{1+\frac{\pi^2}{16}}}{\sqrt{1+\left(\frac{4}{\pi}\right)^2}} = \frac{7}{12}. \qedhere\]

\end{proof}

\section{Conclusion}

We summarize our main result as follows.

\begin{theorem*}
    There exists an orientation-preserving isometry $T$ of the plane such that the distance from $T(P_{n})$ to the curve $r=e^{\frac{4}{\pi}\theta}$ is expressed as
   \[\begin{cases} 
        \ \frac{5}{6}+\mathcal{O}\left(\frac{1}{n}\right)  
        &\text{if $n$ is even} \\[6pt]
        \frac{7}{12}+\mathcal{O}\left(\frac{1}{n}\right) 
        & \text{if $n$ is odd}
     \end{cases}\]
    as $n \to \infty$.
\end{theorem*}

\ind It is also evident from Theorem 9 that the points $T(P_n)$ lies on the ``inner'' side of the spiral; that is, they are located slightly to the left of the curve with respect to the direction of increasing $r$.

\ind It is quite intriguing that the shortest distance between the centers of this structure, constructed from unit regular polygons, and the spiral $r=e^{\frac{4}{\pi}\theta}$ (in an appropriate coordinate system) converges to such a specific rational number.

\ind The term $\frac{(-1)^{n}}{8}$ in the theorem above arises from the inherent nature of this structure. As illustrated in \autoref{fig.2}, the ``bending'' that shapes the structure into a spiral occurs exclusively at the centers of the regular odd-sided polygons, and not at the centers of the even-sided ones. Consequently, if we were to construct a sequence by connecting only the regular odd-sided polygons (i.e., regular $n$-gons for $n=3,5,7,9,\dots$), this alternating term would vanish. Indeed, the following theorem holds:

\begin{theorem*}
    There exists an orientation-preserving isometry $T_{o}$ of the plane such that the distance from $T_o(Q_{n})$ to the curve $r=e^{\frac{4}{\pi}\theta}$ is expressed as
    \[\frac{7}{24}+\mathcal{O}\left(\frac{1}{n}\right)\]
    as $n \to \infty$, where
    \[Q_n=\sum_{k=2}^{n}\frac{1}{2}{\left (\cot\left(\frac{\pi}{2k-1}\right)+\cot{\left(\frac{\pi}{2k+1}\right)}\right )}\ e^{i\pi(H_{2k}-\frac{1}{2}H_k)}.\]
\end{theorem*}

This result can be derived by applying methods analogous to those used in this paper for \eqref{one}.

\section*{Acknowledgment}
I would like to express my sincere gratitude to Dr. Bumtle Kang for kindly providing the endorsement required for my first submission to arXiv. I also extend my thanks to GeoGebra, which was instrumental in inspiring the problem and generating all the figures in this paper.

\providecommand{\bysame}{\leavevmode\hbox to3em{\hrulefill}\thinspace}
\providecommand{\MR}{\relax\ifhmode\unskip\space\fi MR }
\providecommand{\MRhref}[2]{%
  \href{http://www.ams.org/mathscinet-getitem?mr=#1}{#2}
}
\providecommand{\href}[2]{#2}

\end{document}